\documentclass[12pt, oneside]{amsart}      
\usepackage{geometry}                
\usepackage[all]{xy} 
\usepackage{mathrsfs}
\usepackage{setspace}
\usepackage{verbatim}
\usepackage{young} 
\usepackage{graphicx}
\usepackage{amssymb}
\usepackage{epstopdf}
\usepackage{longtable}

\usepackage{tikz}
\usepackage{pgflibraryarrows}
 \usepackage{pgflibrarysnakes}

\usepackage{graphics}
\usepackage{graphicx}

\newcommand{\Id}{\mathrm{Id}}

\newcommand{\End}{\mathrm{End}}

\newcommand{\Stab}{\mathrm{Stab}}
\newcommand{\Spec}{\mathrm{Spec}}

\newtheorem{theorem}{Theorem}

\newtheorem{proposition}[theorem]{Proposition}
\newtheorem{lemma}[theorem]{Lemma}
\newtheorem{corollary}[theorem]{Corollary}

\newenvironment{remark} 
{\par\textbf{Remark}\begin{itshape}\par}
{\end{itshape}}%

\newenvironment{definition} 
{\par\textbf{Definition}\begin{itshape}\par}
{\end{itshape}}%

{\par\textbf{Examples}\begin{itshape}\par}
{\end{itshape}}%

\title[matrix units for $S_d$ and unitary integration]{Matrix Units in the Symmetric Group Algebra, and Unitary Integration }

\author{Timothy Cioppa}
\address{T.C: D\'epartement de Math\'ematique et Statistique, Universit\'e d'Ottawa,
585 King Edward, Ottawa, ON, K1N6N5 Canada}
\email{jciop085@uottawa.ca}
\author{Beno\^\i{}t Collins}
\address{B.C: D\'epartement de Math\'ematique et Statistique, Universit\'e d'Ottawa,
585 King Edward, Ottawa, ON, K1N6N5 Canada,
WPI AIMR, Tohoku, Sendai, 980-8577 Japan
and 
CNRS, Institut Camille Jordan Universit\'e  Lyon 1,
France}
\email{bcollins@uottawa.ca}
                                         
\begin{document}

\maketitle

\begin{abstract}

In this paper, we establish an explicit isomorphism between the symmetric group algebra 
$\mathbb{C}[S_d]$ and the path algebra of the Young graph $\mathbb{Y}_d$, by expressing the matrix elements
$E_{T,S}^{\lambda}$ as a linear combination of group elements.

We then investigate applications of this result. As a main application, we obtain new
formulas, alternative to Weingarten calculus, to integrate polynomials with respect to
the Haar measure on the unitary group.
In particular, we obtain a closed formula for the law of moments of the first $k$ rows of
the unitary group $U_n$, uniform in $n\geq k$.

\end{abstract}

\section{Introduction}
\label{sec:Introduction}

A fundamental result from representation theory is that the group algebra of a finite 
group $G$ is isomorphic to a matrix algebra (see \cite{Tolli}).  In particular, a group algebra is spanned by a family of matrix units $E_{T,S}^{\lambda}$ where $\lambda \in \hat{G}$ and for which the 
multiplication follows the normal rules: 
$$E_{T,S}^{\lambda}E_{RM}^{\beta} = \delta_{S,R}\delta_{\lambda,\beta}E_{T,M}^{\lambda}.$$ 
It is a difficult problem in general to compute these matrix elements 
in terms of the abstract elements of the group algebra.  In this paper, we provide such a formula for the symmetric group algebra, and discuss some applications of these formula.  

Our first main result can be stated as follows:  the elements 
$E_{T,S}^{\lambda} := E_T^{\lambda}\sigma_{S,T}E_S^{\lambda}$ form,
up to a non-zero multiplicative factor,
a set of matrix units for the group algebra $\mathbb{C}[S_d]$, where $T$ and $S$ 
denote standard fillings of the Young diagram $\lambda$, 
$E_T^{\lambda}$ is a shorthand notation for $E_{TT}^{\lambda}$, 
and 
$\sigma_{S,T}$ is the unique permutation transforming $S$
into $T$.
Using available formulas for $E_T^{\lambda}$ (\cite{Tolli}, Theorem 3.4.11) 
we obtain explicit formulas for the matrix units as linear combinations of permutations. 
Similar formulas have been obtained in \cite{Ram} 
in a more general setup, in a recursive way, rather than directly as a combination of group elements.

We will discuss two applications of these formulas. The first result concerns formulas for inclusion of matrix 
units $E_{T,S}^{\lambda} \in \mathbb{C}[S_d] \subset \mathbb{C}[S_{d+1}]$, and conversely formulas for 
the projection of matrix units in $\mathbb{C}[S_{d+1}]$ onto $\mathbb{C}[S_d]$. Not surprisingly, these 
formulas depend of the structure of the Young graph, which encodes complete information on the 
irreducible representations of the symmetric groups of all order. 

A second and more important byproduct is an application to the theory of integration over the unitary group
with respect to the Haar measure.

An
explicit formula for the integral of
a polynomial against the Haar measure
has remained elusive until quite recently (see \cite{co,cs1}). 
We propose here a completely new method for 
performing this task.
Our main result can be stated as follows:

\begin{theorem}
\label{theorem1}
For all $d$-tuples of indices $I=(i_1,\ldots ,i_d),J=(j_1,\ldots , j_d ),K=(k_1,\ldots , k_d),L=(l_1,\ldots , l_d)$, 
the following formula holds true:
$$\int_{U_n}u_{i_1j_1}\ldots u_{i_dj_d}\overline u_{k_1l_1}\ldots \overline u_{k_dl_d}d\mu (U) 
=  \sum_{\lambda \vdash d ; S,T\in Stab(\lambda), l(\lambda) \leq n} {\dfrac{\langle e_{J,L},E_{S,T}^{\lambda}\rangle \langle E_{S,T}^{\lambda},e_{I,K} \rangle}{||E_{S,T}^{\lambda}||^{2}}},$$
where
$E_{S,T}^{\lambda}$ are 
multiples of
matrix units in $\mathbb{C}_n[S_d] \subset M_n^{\otimes d}$, as defined in Theorem \ref{theorem4}, and $e_{I,J} = e_{i_1, j_1} \otimes \dots \otimes e_{i_d, j_d}$ are the standard matrix units in $M_n^{\otimes d}$.
\end{theorem}

The paper is organized as follows: in section ~\ref{sec:Reminder} 
we begin with a review of the necessary background concerning Young diagrams, the Young graph $\mathbb{Y}$, and the representation theory of the symmetric groups $S_d$, $d \geq 1$.

In section \ref{sec:matrixunits}
we derive explicit formulas for the matrix units $E_{T,S}^{\lambda} \in \mathbb{C}[S_d]$, and describe the behaviour of these matrix units with respect to projection and inclusion between symmetric groups of different orders.

Finally, in section \ref{sec:integrals}, we show how these matrix formulas can be used in the calculation of polynomial integrals over the 
unitary group $U_n \subset GL_n(\mathbb{C})$ with respect to the Haar measure $\mu = \mu_n$.

\section{Reminder: Young Diagrams and Irreducible $S_d$ Modules}
\label{sec:Reminder}

In order to construct matrix elements for $\mathbb{C}[S_d]$, we first need to recall the construction of 
irreducible $S_d$ modules. We follow the lines of 
Vershik and Okounkov in \cite{Vershik}, which contains all the details of this section. The representation 
theory of the symmetric groups relies primarily on the combinatorial structure of Young diagrams. 

A \emph{Young diagram} $\lambda$ is defined as 
a finite non-increasing sequence of integers $\lambda_1\geq \ldots \geq \lambda_k > 0$. 
The \emph{size} of the diagram $\lambda$ is given by $|\lambda| = \sum_i \lambda_i$, and we 
use the notation $\lambda \vdash d$ to mean $\lambda$ is a Young diagram of size $d$. The number $k$ is known as the \emph{length of $\lambda$} and is denoted $l(\lambda)$. Young diagrams have a useful graphical representation as a collection of boxes, with three conventional representations shown below (English, French, and Russian). 
For example, the graphical representation of the diagram $\lambda = (4,3,3,2,1)$ is shown below. 

\begin{figure}
	\begin{center}
		\includegraphics[scale=0.8]{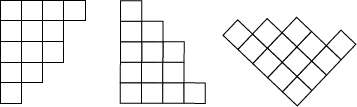}
		\caption{Figure 1: three well known ways for graphically representing a Young diagram.}
		\label{fig1}
	\end{center}
\end{figure}

In what follows, we
adopt the English convention for drawing Young diagrams. The Young graph $\mathbb{Y}$ 
(the first 4 levels of which is shown below in figure \ref{fig2}) 
 is the infinite directed graph whose vertex set is the set of all Young diagrams, and where a diagram of size 
 $d$ is connected to one of size $d+1$ if the two differ by exactly one box. The truncated Young graph $\mathbb{Y}_d$ consists only of the first $d$ levels of $\mathbb{Y}$. Given any diagram 
 $\lambda$, we can consider the set  
 $\Stab(\lambda )$
 of all paths in the Young graph starting at the unique block of size $1$ and ending at $\lambda$. 
 Equivalently, such a path in the Young graph corresponds in a natural way to a filling of the boxes of 
 $\lambda$ with the numbers $1,\dots, d$ so that the numbers are increasing along every row and column 
 of the diagram. Such a filling is called \emph{standard}, which explains the usage of $\Stab(\lambda$), 
 denoting the set of standard tableaux. To each standard filling $T$ of the diagram $\lambda$ of size $d$,  
 we can associate the \emph{content vector} $c(T) = (a_1(T),\dots, a_d(T))$, where $a_i(T)$ is the difference between the $x$ and $y$ coordinates of the $i^{th}$ box added according to the filling $T$. 
 This vector encodes important information about the representations of $S_d$ which we will discuss later. 

\begin{figure}
	\begin{center}
		\includegraphics[scale=0.8]{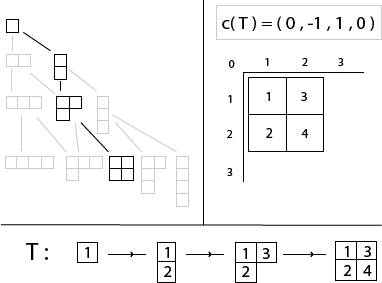}
		\caption{Figure 2: The first four levels of $\mathbb{Y}$}
		\label{fig2}
	\end{center}
\end{figure}

Using the Young graph we can construct all irreducible $S_d$ modules as follows: 
let $V_{\lambda}$ be the free complex vector space with orthonormal basis consisting of elements 
$w_T$ where $T \in \Stab(\lambda)$. We define the action of $S_d$ on this space by reducing to the 
case of Coxeter transpositions $s_i=(i,i+1)$, which generate the symmetric group $S_d$. 
If $s_i$ is a Coxeter transposition, and $T\in \Stab(\lambda)$, we consider the following two cases:
\begin{enumerate}
	\item If the filling $s_iT$ (which is obtained by replacing $k$ with $s_i(k)$ in the filled diagram 
	$\lambda$) is not standard, define $s_i(w_T)$ to be $w_T$ if $i$ and $i+1$ lie in the same row of $T$, 
	and -$w_T$ if they lie in the same column of $T$. 
	\item If the filling $S=s_iT$ is also standard, define $s_i(w_T) = \dfrac{1}{r}w_T+\sqrt{1-\dfrac{1}{r^2}}w_S$ 
	and $s_i(w_S) = \sqrt{1-\dfrac{1}{r^2}}w_T-\dfrac{1}{r}w_S$, where r is the \emph{axial distance} defined by $r=r_i(T)= a_{i+1}(T)-a_i(T)$.
\end{enumerate}

The following result can be found in \cite{Vershik}.

\begin{theorem}
\label{theorem2}
Indexing over all partitions $\lambda$ of $d$ boxes, the set of $V_{\lambda}$ defined above constitute a 
full set of pairwise non isomorphic irreducible representations of $S_d$. Hence, we have a $*-$algebra isomorphism 
$\mathbb{C}[S_d] = \bigoplus_{\lambda \vdash d} \End(V_{\lambda})$.
\end{theorem}

The combinatorial structure of Young diagrams is equivalent to the algebraic structure of the chain 
$\mathbb{C}[S_1] \subset \mathbb{C}[S_2] \subset \mathbb{C}[S_3] \subset ... $ of symmetric group algebras. 
This algebraic structure is encoded in the so called \emph{Bratteli diagram} of the chain of algebras. 
This is the directed graph whose vertices consist of all isomorphism classes of irreducible $S_d$ modules, 
for all integers $d \geq 1$. We connect two isomorphism classes $\lambda$ of $S_d$ and $\beta$ of $S_{d+1}$  by exactly 
$k$ directed edges (from $\lambda$ to $\beta$) when the multiplicity of $\lambda$ in the restriction of 
$\beta$ to $S_d$ is $k$. The following is a summary of the important properties of the Bratteli diagram of 
the symmetric groups, and can be found in \cite{Vershik}:

\begin{theorem}[\cite{Vershik}]
\label{theorem3}
The Bratteli diagram of the symmetric groups is (graph theoretically) isomorphic to the Young graph, 
whose vertices are Young diagrams of size $d$ (for all natural numbers $d \geq 1$), and whose edges are determined 
by the inclusion of Young diagrams of size $d$ into Young diagrams of size $d+1$. 
\end{theorem}

We obtain a canonical basis, called the \emph{Young Basis} of an irreducible $S_d$ module 
$V_\lambda$ as follows: 
for each path $T$ in the Young graph (Bratteli diagram) from the unique $S_1$ module $\lambda_1$ to 
$\lambda$, say $T = \lambda_1 \to \lambda_2 \to ... \to \lambda$, 
choose a vector $v_T$ in $V_\lambda$ such that $v_T \in V_{\lambda_i}$ for all $i$. 
Doing so for each path in the Young graph gives a basis indexed by paths in the Young graph from 
$\lambda_1$ to $\lambda$, or equivalently, 
indexed by standard fillings of the diagram $\lambda$. The simplicity of the Bratteli diagram ensures that this 
basis is uniquely determined up to scalar multiplication. One of the most important properties 
of this basis  is its connection to the \emph{Gelfand Tsetlin algebra}, the subalgebra 
$GZ(d) \subset \mathbb{C}[S_d]$ generated by the centers $Z(\mathbb{C}[S_1]),..., Z(\mathbb{C}[S_d])$. 
Because of the simplicity of the Bratteli diagram, this subalgebra contains all projections from 
irreducible $S_{k+1}$ modules to irreducible $S_k$ modules. 
Hence this subalgebra contains the maximally abelian subalgebra of operators diagonal on the Young basis. 
Being abelian itself, this implies that 
the GZ algebra consists entirely of operators diagonal on the Young basis.  

In order to study the structure of the GZ-algebra, we consider the elements 
$X_i \in \mathbb{C}[S_d]$ defined as $X_1=0, X_2 = (1,2), X_3 = (1,3)+(2,3),...$. These are called the 
\emph{Young-Jucys Murphy elements} of the group algebra $\mathbb{C}[S_d]$, and their spectrum on the 
Young basis is the key to the analysis of the Bratteli diagram described above. The following facts summarize 
the relationship between $GZ(d)$, the Young basis, and the Jucys Murphy elements, and can be found in \cite{Vershik}:

	\begin{enumerate}
		\item $GZ(d)$ is the maximally abelian subalgebra of $\mathbb{C}[S_d]$ consisting of those 
		elements whose Fourier transform is a 
		diagonal operator with respect to the Young basis. 
		\item $GZ(d)$ is generated by the elements $X_1,X_2,...,X_d$. 
		\item The Young basis of a representation $V_{\lambda}$ is completely determined by the 
		eigenvalues of $X_1,...,X_d$ acting on it. 
	\end{enumerate}

For any Young vector $v_T \in V_\lambda$, denote the eigenvalues of $X_1,X_2,...,X_d$ on $v_T$ by the 
vector $c(T) = (a_1(T),...,a_d(T))$. The set of all such vectors is called 
$\Spec(d)$ in \cite{Vershik}. It can be shown that for a path $T$ in the Young graph, $c(T)$ is in fact the content 
vector of the standard filling associated to $T$. 
This is a highly important and nontrivial fact, and a complete proof can be found in \cite{Vershik} or \cite{Tolli}. 

Because the Young basis is unique up to scalar multiplication, it is possible (see \cite{Vershik}) to chose 
normalized coefficients such that the Coxeter generators act on the 
normalized basis $\{ w_T | T  \in \Stab(\lambda) \}$ according the rules (1) and (2) from the beginning of this section. 

Given the decomposition $\mathbb{C}[S_d] \cong \bigoplus_{\lambda \vdash d} \End(V_{\lambda})$, we obtain a set of 
matrix units $E_{T,S}^{\lambda}$ for $\mathbb{C}[S_d]$. In particular, $E_{T,S}^{\lambda}$ is the operator 
defined by $E_{T,S}^{\lambda} w_{R} = \delta_{R,S}w_T$ where $R$ is another standard filling of $\lambda$. 
We will denote the 
minimal projection $E_{T,T}^{\lambda}$ simply by $E_T^{\lambda}$. By our choice of basis, these elements 
lie in GZ($d$), and so they can be written as polynomials 
in the YJM elements. In the next section we will discuss the polynomials $p_T^{\lambda}$ such that 
$E_T^{\lambda}= p_T^{\lambda}(X_1,...,X_d)$, and we will use these to obtain 
explicit formulas for $E_{T,S}^{\lambda}$ as elements of $\mathbb{C}[S_d]$.

\section{Matrix Units in $\mathbb{C}[S_d]$}
\label{sec:matrixunits}

\subsection{Minimal Projections}
\label{subsec:centralprojections}

Before giving formulas for the general matrix unit, $E_{T,S}^{\lambda}$, we will first recall the formulas 
for the minimal projections 
$E_T^{\lambda} \in \mathbb{C}[S_d]$ as polynomials in the Young Jucys Murphy elements.

If $T\in \Stab(\lambda)$ is a standard filling of $\lambda$, denote by $\overline{T}$ the standard tableau of 
$d-1$ boxes obtained by removing the box containing 
$d$ from $T$, and denote by $\overline{\lambda}$ the shape of $\overline{T}$. 

\begin{proposition}
\label{prop1}
Let T and $\overline{T}$ be as above. Then as a polynomial in the YJM elements, we have
\begin{equation*}
	E_T^{\lambda} = E_{\overline{T}}^{\overline{\lambda}} (\prod_{S \ne T, \overline{S}=\overline{T}} \frac{(a_d(S)-X_d)}{(a_d(S)-a_d(T))})
\end{equation*}
\end{proposition}

This result can be found in \cite{Tolli}, and follows from the fact that the right hand side behaves as a 
minimal projection should, in that it maps $v_T$ to $v_T$, and maps all other basis vectors to zero. 
This formula is similar in spirit to the formula for the spectral projections of  a self adjoint operator 
$A \in M_n(\mathbb{C})$ with a simple spectrum $\lambda_i \neq \lambda_j$, $i \neq j$(i.e. no repeated eigenvalues):

\begin{equation*}
	E_{\lambda_i}= \prod_{i \neq j} \dfrac{(A-\lambda_j\Id)}{(\lambda_i - \lambda_j)}
\end{equation*}
 
As a simple example, consider the path $T$ shown in figure \ref{fig2}. Using proposition \ref{prop1} we have $E_T^{\lambda} = E_{\overline{T}}^{\overline{\lambda}}(\dfrac{1}{4}(2 - X_4)(2+X_4))$. Repeating for $E_{\overline{T}}^{\overline{\lambda}}$ we have $E_{\overline{T}}^{\overline{\lambda}} = E_{\overline{\overline{T}}}^{\overline{\overline{\lambda}}} \dfrac{1}{3}(2+X_3)$. Finally $E_{\overline{\overline{T}}}^{\overline{\overline{\lambda}}} = \dfrac{1}{2}(1 - X_2)$. This gives $E_T^{\lambda} = \dfrac{1}{6}(2+X_4)(2-X_4)(2+X_3)(1-X_2)$.
 
From a computational point of view, specifically for calculating unitary integrals, we are free to use multiples of minimal projections: 
$$\tilde E_{\lambda_i}= \prod_{i \neq j} {(A-\lambda_j\Id)},$$
which are computationally less expensive. However, for all integration formulas we present in this paper, we will assume that the minimal projections have been normalized as in proposition \ref{prop1}.

\subsection{Main result}
\label{subsec:mainresult}

In this section we present explicit formulas for the matrix units $E_{T,S}^{\lambda}$.
Since
 the formulas
 rely on the pointwise action of the symmetric group on fillings of a given tableau $\lambda$, we 
 recall that if $T$ is a filling of $\lambda$ and $\sigma \in S_d$, then $\sigma T$ is the filling of $\lambda$ 
 obtained by replacing $i$ in $T$ by $\sigma(i)$. This defines a transitive action of $S_d$ on the set of 
 fillings of a given diagram $\lambda$. 

We start with the following:

\begin{proposition}
\label{prop2}
Let $T$ and $S$ be standard fillings of the digram $\lambda$, and suppose 
$T=s_iS$. Then $E_{T,S}^{\lambda} = \sqrt{\dfrac{r^2}{r^2-1}}(E_T^{\lambda}s_iE_S^{\lambda})$, 
where $r$ is the $i$th axial distance of the filling $T$. 
\end{proposition}
\begin{proof}
Using formula (2) from page 3, we see that the right hand side of the equality sends $w_S$ to $w_T$ 
and sends all other basic elements to zero. Thus it must be $E_{T,S}^{\lambda}$.
\end{proof}
In order to generalize the previous result to arbitrary standard fillings $T$ and $S$ of a diagram $\lambda$, 
we need the notion of an \emph{admissible transposition} for the diagram $T$. 

A Coxeter generator $s_i = (i, i+1)$ is called \emph{admissible for T} if the filling $S = s_i T$ is also standard. 
Any two standard fillings of the same diagram can be transformed into one another by a sequence of 
admissible Coxeter generators. The minimal number of admissible Coxeter transpositions required is 
called the \emph{Coxeter distance between T and S}, and will be denoted by $d(T,S)$. Given a diagram 
$\lambda$, the Coxeter distance $d$ is a well defined metric on the set of paths Stab($\lambda$).

\begin{theorem} 
\label{theorem4}
Let $T,S \in Stab(\lambda)$. Let $\sigma$ be the permutation sending $T$ to $S$. 
Then, there exists $c \neq 0$ such that 
$E_{T,S}^{\lambda} = cE_T^{\lambda}\sigma^{-1}E_S^{\lambda}$. 
\end{theorem}

\begin{proof}
The case when $T$ and $S$ differ by a single Coxeter generator was proven above, so we shall prove the result 
by induction on the Coxeter distance $d(T,S)$. If said distance is $k+1$ we can send $T$ to $S$ via 
$R \in \Stab(\lambda)$, where $d(T,R) = k$ and $d(R,S)=1$. Suppose that $T$ is sent to $R$ with $\sigma$ 
and $R$ is sent to $S$ with $s_i$. The following calculation gives the desired result:

\begin{eqnarray*} 
	E_T^{\lambda}\sigma^{-1}s_i E_S^{\lambda} &=&  E_T^{\lambda}\sigma^{-1}\Id s_i E_S^{\lambda} \\ 
	&=&E_T^{\lambda}\sigma^{-1}E_R^{\lambda}s_i E_S^{\lambda} + \sum_{L\neq R}E_T^{\lambda} \sigma^{-1}E_L^{\lambda}s_i E_S^{\lambda} \\
	&=& E_T^{\lambda}\sigma^{-1}E_R^{\lambda}s_i E_S^{\lambda} +E_T^{\lambda}\sigma^{-1}E_S^{\lambda}s_i E_S^{\lambda} 
\end{eqnarray*}
Both summands are proportional to $E_{T,S}^{\lambda}$, 
but
it is possible the whole sum is zero. 
We must prove that $E_T^{\lambda}(\sigma^{-1}w_S) = 0$. However, since $\sigma^{-1}$ can, by assumption, 
be written as the product of $k$ 
Coxeter generators, we have that $\sigma^{-1} w_S$ lies in the span of those $w_Q$ with $d(S,Q) \leq k$. 
Since, in particular, $d(S,T) = k+1$ we have that $E_T^{\lambda} \sigma^{-1}w_S = 0$ and so the theorem holds. 
\end{proof}

The non-zero scalar $c$ actually has a closed form which is a clear generalization of proposition \ref{prop2}.
Given the permutation $\sigma$ above, decompose it into the minimal number  of  
Coxeter generators $\sigma = s_{i_k}\dots s_{i_1}$, so that $T \rightarrow T_1 \rightarrow \dots \rightarrow S$ is transformed into $S$ 
via the generators $s_{i_1},\dots,s_{i_k}$. Let $r_m = a_{i_m+1}(T_m) - a_{i_m}(T_m)$. 

\begin{proposition}
\label{prop3}

The non-zero scalar $c$ appearing in Theorem \ref{theorem4}
is given by $\sqrt{\prod_i(\dfrac{r_i^2}{r_i^2-1})}$.
\end{proposition}

\begin{remark}
While it is possible to have two different minimal decompositions of a given permutation into 
admissible transpositions, the constant in the previous proposition does not depend on the 
particular decomposition. We are unaware of a direct combinatorial proof of this fact.
\end{remark}

For the integration formulas in section 4, we will ignore the nonzero scalar 
$c$ and simply use the orthogonal (but non-normalized) basis
$E_T^{\lambda}\sigma_{S,T}E_S^{\lambda}$ of $\mathbb{C}[S_d]$.

\subsection{Comparison with known Formulas}
\label{subsec:comparison}

In \cite{Ram}, the authors construct a family of matrix units for the type 
$A$ Hecke algebras $H_d(q)$, which gives the symmetric group algebra $\mathbb{C}[S_d]$ when $q = 1$. 
In particular, these algebras have generators $g_1, \dots, g_{d-1}$ 
satisfying the relations

\begin{enumerate}
\item
$g_i g_{i+1}g_i = g_{i+1}g_i g_{i+1}$ for all $1 \leq i \leq d-2$
\item
 $g_i g_j = g_j g_i$ for all $|i-j| \geq 2$
\item
 $g_i^2 = (q-1)g_i + q$ for $1 \leq i \leq d-1$.
\end{enumerate}

For almost all $q$ this algebra is isomorphic to $\mathbb{C}[S_d]$, replacing $g_i$ with $(i, i+1)$. 
The construction of the matrix units given in \cite{Ram} are
more general than our construction, 
as they apply to $H_d(q)$ for all $q,d$. However, they are computationally far more expensive than the 
formulas we present in this section. 
Indeed, in order to construct $E_{T,S}^{\lambda}$ we need only 
construct $E_T, E_S$ and then multiply by the permutation $\sigma_{S,T}$ presented in this section.  
The presentation in \cite{Ram} cannot rely on any such element $\sigma_{S,T}$.

Actually, the formulas in \cite{Ram} have a normalizing constant in front, but for our purposes (in particular for 
calculating unitary integrals in the next section) the normalizing constants are not needed, only the orthogonality of 
the matrix units. The formulas in this section and in \cite{Ram} of course give the same results, but as we are 
concerned with unitary integrals (and hence with the symmetric group algebra) and with computation, our 
formulas are an improvement over those in \cite{Ram}, as they require fewer applications of induction on the 
size of the diagram $\lambda$.

\subsection{Inclusion Rules for $\mathbb{C}[S_d] \subset \mathbb{C}[S_{d+1}]$}
\label{subsec:inclusion}

Having formulas for all $E_{T,S}^{\lambda}$ in  $\mathbb{C}[S_d]$ leads to a natural question involving the 
inclusion of $\mathbb{C}[S_d] \subset \mathbb{C}[S_{d+1}]$. Recall that $S_d \subset S_{d+1}$, where 
a permutation of $d$ elements is viewed as a permutation of $d+1$ elements fixing the $(d+1)^{st}$. This 
extends to an inclusion of $*$ - algebras $\mathbb{C}[S_d] \subset \mathbb{C}[S_{d+1}]$. 

Given a Young diagram $\lambda \vdash d$, and given $T$,$S$ $\in \Stab(\lambda)$, we have an element  
$E_{T,S}^{\lambda}$ in  $\mathbb{C}[S_d]$. As an element of $\mathbb{C}[S_{d+1}]$, there is a 
unique decomposition

\begin{equation*}
	E_{T,S}^{\lambda}= \sum_{\beta \vdash d+1} \sum_{R,M \in Stab(\beta)} \alpha_{T,S}^{\beta}(R,M)E_{R,M}^{\beta}, 
\end{equation*}
where $\alpha_{T,S}^{\beta}(R,M)$ are complex numbers depending on $\beta,T, S, R$, and $M$. We would 
like a characterization of these coefficients.  

\begin{theorem}
\label{theorem5}
Let $\lambda \vdash d$ and T $\in$ Stab($\lambda$). Then as an element of 
$\mathbb{C}[S_{d+1}]$, $E_T^{\lambda}$ is projection onto the subspace $\bigoplus_{S}\mathbb{C}w_S$ 
spanned by all $w_S$ such that $\overline{S}=T$. Hence, 
\begin{equation*}
	E_T^{\lambda}= \sum_{S} E_S^{\beta} 
\end{equation*}
where we sum over all S such that $\overline{S}=T$ and $\overline{\beta}=\lambda$. 
\end{theorem}

\begin{proof}
We need to verify that if $\overline{S}=T$, then $E_T^{\lambda}(w_S)=w_S$, otherwise 
$E_T^{\lambda}(w_S)=0$. Write $T=\lambda_1 \to \lambda_2 \to... \to \lambda_d=\lambda$, and let 
$T_i$ be the path $\lambda_1 \to...\to \lambda_i$. Then we can write
\begin{equation*}
	E_T^{\lambda}= \prod_{i=2}^d \prod_{\overline{S}=\overline{T_i}, S\neq T_i}\frac{(a_i(S)-X_i)}{(a_i(S)-a_i(T_i))}
\end{equation*}
If $\overline{S}=T$ then $a_i(T)=a_i(S)$ for each $i=1,...,d$, and hence the polynomial above will send 
$w_S$ to $w_S$ as required. If $\overline{S}\neq T$ then the polynomial above will send $w_S$ to $0$ 
by construction. 
\end{proof}

\begin{lemma}
\label{lemma1}
Suppose $T,S \in Stab(\lambda)$, where $\lambda \vdash d$. Then, as an element of $\mathbb{C}[S_{d+1}]$, 
\begin{equation*}
	E_{T,S}^{\lambda}= \sum_{R,M} E_{R,M}^{\beta} 
\end{equation*}
where we sum over all $R$ and $M$ of the same shape $\beta$ such that $\overline{R}=T$ and $\overline{M}=S$. 
\end{lemma}

\begin{proof}
We begin with the case when $T$ and $S$ differ by a Coxeter transposition $s_i$. 
Whenever $\overline{M}=S$, the right hand side of the equation above sends $w_M$ to $w_R$ for the 
unique $R$ of the same shape as $M$ with $\overline{R}=T$. Otherwise, $w_M$ is sent to $0$. 
Hence we need to verify that the left hand side does the same. We can right $E_{T,S}^{\lambda}$ as 
$c E_T^{\lambda}s_i E_S^{\lambda}$ where $c$ is the nonzero constant from proposition \ref{prop2}.
Since $s_iS=T$ implies $s_iM=R$, we have that $c E_T^{\lambda}s_i E_S^{\lambda}(w_M)=w_R$ 
when we are in the first situation above, otherwise  $c E_T^{\lambda}s_i E_S^{\lambda}(w_M)=0$, as required. 

For the general case, we proceed by induction on the Coxeter distance between $T$ and $S$. The case 
when $d(T,S)=1$ was proven above. Now suppose $d(T,S)=k+1$. Choose $R$ such that $d(T,R)=k$ 
and $d(R,S)=1$. We have $E_{T,S}^{\lambda}= E_{T,R}^{\lambda}E_{R,S}^{\lambda}$. By induction, the 
formula holds for each of $E_{T,R}^{\lambda}$ and $E_{R,S}^{\lambda}$, and multiplying out gives 
the required result. 
\end{proof}

\vspace{0.3cm}

The dual question to that of the inclusion of matrix units $E_{T,S}^{\lambda} \in \mathbb{C}[S_d] \subset \mathbb{C}[S_{d+1}]$ is the following: how do the matrix units 
$E_{T,S}^{\lambda}$ decompose upon restriction from $\mathbb{C}[S_{d+1}]$ to $\mathbb{C}[S_{d}]$. 
To be more precise, we define the map $\mathbb{E}: \mathbb{C}[S_{d+1}] \rightarrow \mathbb{C}[S_{d}]$ 
by stipulating that for $\sigma \in S_{d+1}$, $\mathbb{E}(\sigma)=\sigma$ if $\sigma \in S_d$ and 0 otherwise. 
This map is just projection from $\mathbb{C}[S_{d+1}]$ to $\mathbb{C}[S_{d}]$, which we call the 
\emph{conditional expectation}. Our question then is the following: 
for $E_{T,S}^{\lambda} \in \mathbb{C}[S_{d+1}]$, what are the coefficients of 
$\mathbb{E}(E_{T,S}^{\lambda}) = \sum_{\beta \vdash d} \sum_{R,M \in Stab(\beta)} \alpha_{T,S}^{\beta}(R,M)E_{R,M}^{\beta}$ as an element of $\mathbb{C}[S_d]$? 

\begin{theorem}
\label{theorem6}
If $\overline{S}$ and $\overline{T}$ have the same shape $\beta$, then $\mathbb{E}(E_{T,S}^{\lambda}) = \dfrac{\dim(V_{\lambda})}{d\dim(V_{\beta})}E_{\overline{T},\overline{S}}^{\beta}$. Otherwise, $\mathbb{E}(E_{T,S}^{\lambda}) = 0$.
\end{theorem}

The proof of this, and a more general result can be found in \cite{Ram}, where the author looks at the setting of a 
chain $A_1 \subset A_2 \subset \dots$ of finite dimensional, semi-simple complex algebras whose Bratteli 
diagram has simple branching rules. As the symmetric group algebras form such a chain, the formulas 
in \cite{Ram} apply. However, they are computationally far more complex, and the proofs don't take 
advantage of the combinatorial structure of the Bratteli diagram that the results in \cite{Vershik} afford.

\section{Application to Unitary matrix integrals}
\label{sec:integrals}

\subsection{Algebraic preliminaries}

In this section, we are interested in the following problem. 
Let $\mu$ be the normalized Haar measure on the unitary group $U_n$ and 
$u_{ij}:U_n\to \mathbb{C}$ the $ij$ coordinate map.
We are interested in computing all moments of $\mu$, or equivalently, all integrals
$$\int_{U_n}u_{i_1j_1}\ldots u_{i_dj_d}\overline u_{k_1l_1}\ldots \overline u_{k_dl_d}d\mu (U)$$

Let us start with the following tensor reformulation:

\begin{lemma}
\label{lemma2}
$$\int_{U_n}u_{i_1j_1}\ldots u_{i_dj_d}\overline u_{k_1l_1}\ldots \overline u_{k_dl_d}d\mu (U)=
\int_{U_n} tr(e_{K,I}U^{\otimes d}e_{J,L}U^{*\otimes d})d\mu (U)$$
\end{lemma}

In order to calculate these integrals explicitly, we 
use an algebraic result connecting the representation theory of the symmetric groups with those 
of the unitary groups. Recall that the vector space $(\mathbb{C}^n)^{\otimes d}$ is both a $S_d$ module and a $U_n$ module. The action of $S_d$ is given by the homomorphism $p_n^d : \mathbb{C}[S_d] \rightarrow \End((\mathbb{C}^n)^{\otimes d})$, with $p_n^d(\sigma)(v_1 \otimes \dots \otimes v_d) = v_{\sigma^{-1}(1)} \otimes \dots v_{\sigma^{-1}(d)}$. The action of $U_n$ is given by the homomorphism $\rho : U_n \rightarrow GL((\mathbb{C}^n)^{\otimes d})$, with $\rho(U)(v_1 \otimes \dots \otimes v_d) = (Uv_1) \otimes \dots \otimes(Uv_d)$. It is clear that these actions commute with one another, so that $(\mathbb{C}^n)^{\otimes d}$ is in fact a $\mathbb{C}[S_d] \times U_n$ module with the following structure:

\begin{theorem}[Schur-Weyl Duality, \cite{Tolli}, Theorem 8.2.10]
\label{thm:schur-weyl}
\label{theorem7}
The multiplicity free decomposition of $(\mathbb{C}^n)^{\otimes d}$ into irreducible 
$S_d \times U_n$ modules is given by $\bigoplus_{l(\lambda) \leq n} V_{\lambda}\otimes U^{\lambda}$, 
where $U^{\lambda} = \End_{S_d}(V_{\lambda}, (\mathbb{C}^n)^{\otimes d})$, and $V_{\lambda}$ is the 
irreducible representation of $S_d$ corresponding to the partition $\lambda$ of $d$.  
\end{theorem}

Note that restricting to the subgroups ${1} \times U_n \cong U_n$ and $S_d \times {1} \cong S_d$ gives the decomposition of $(\mathbb{C}^n)^{\otimes d}$ as a $U_n$ and $S_d$ module. In particular, the multiplicity of $V_{\lambda}$ in $(\mathbb{C}^n)^{\otimes d}$ is zero for $l(\lambda) > n$. This fact will be used later in calculating integrals over sub-rectangles in $U_n$. 

\subsection{A reminder of Weingarten calculus}

Before supplying a new integration formula, let us recall the existing integration technique, known as Weingarten calculus. 
The idea goes back to 
 \cite{We}.

\begin{definition}
Consider the set $P(d)^U$ of pair partitions of $\{1,\ldots ,2d\}$ linking an element $<d$ with an element $\geq d$.
\begin{enumerate}
\item We plug multi-indices $i=(i_1,\ldots ,i_{2d})$ into partitions $p\in P(d)^U$, and we set $\delta_{pi}=1$ 
if all blocks of $p$ contain equal indices of $i$, and $\delta_{pi}=0$ if not.
\item The Gram matrix of partitions (of index $n\geq 4$) is given by $G_{n,d}(p,q)=n^{|p\vee q|}$, where 
$\vee$ is the set-theoretic sup, and $|.|$ is the number of blocks.
\item The Weingarten matrix $W_{n,d}$ is the inverse of $G_{n,d}$.
\end{enumerate}
\end{definition}

We can view elements of $P(d)^U$ as permutations in $S_d$ 
acting on $\End((\mathbb{C}^n)^{\otimes d})$.
The Gram matrix of this basis with respect to the scalar product induced by the non-normalized canonical trace
is nothing but $G_{k,n}$, as shown by the following computation:
$$<p,q>=\sum_i\delta_{pi}\delta_{qi}=\sum_i\delta_{p\vee q,i}=n^{|p\vee q|}$$

With these notations, we have the following result.

\begin{theorem}\cite{co,cs1}
\label{theorem8}
The Haar functional is given by
$$\int u_{i_1j_1}\ldots u_{i_kj_k}=\sum_{pq}\delta_{pi}\delta_{qj}W_{n,k}(p,q)$$
where the sum is over all pairs of diagrams $p,q\in NC(k)$.
\end{theorem}

The original proof can be found in 
 \cite{co,cs1}, and simplified proofs are available in the more general quantum group setup
 \cite{bc1,bc2,bc3}. See also \cite{cm}.

\subsection{A new integration formula}
\label{subsec:integrationformula}

Theorem \ref{thm:schur-weyl}
 implies that the map $p_n^d$ is injective on the subalgebra $\mathbb{C}_n[S_d] \cong \bigoplus_{l(\lambda) \leq n} \End(V_\lambda)$, and hence we will view $\mathbb{C}_n[S_d]$ as being contained in $\End((\mathbb{C}^n)^{\otimes d})$ via the map $p_n^d$. Further, the subalgebra $\mathbb{C}_n[S_d] \cong \bigoplus_{l(\lambda) \leq n} \End(V_\lambda)$ is contained in the centralizer $\End_{U_n}((\mathbb{C}^n)^{\otimes d})$. Theorem 7 implies that we actually have equality, giving the following proposition:

\begin{proposition}
\label{prop4}
For any $A\in M_n^{\otimes d}$, we have
$$\int_{U_n} U^{\otimes d} A  (U^{-1})^{\otimes d}d\mu (U)=\mathbb{E}(A)$$
where $\mathbb{E}$ is the orthogonal projection with respect to the Hilbert Schmidt norm of $M_n^{\otimes d}$ onto the sub algebra $\mathbb{C}_n[S_d] \subset \End((\mathbb{C}^n)^{\otimes d})$.

\end{proposition}

Strictly speaking the projection above is orthogonal with respect to the $U_n$ invariant 
inner product on $M_n^{\otimes d}$, which in our case happens to be the Hilbert Schmidt 
inner product. 

Proposition \ref{prop4},  
together with the following elementary fact, will allow us to obtain specific formulas for polynomial 
integrals. The orthogonal projection above is normally referred to as \emph{conditional expectation}. 

\begin{proposition}
\label{prop5}
$$\mathbb{E}(A)=\sum_{\lambda \vdash d ; S,T\in Stab(\lambda), l(\lambda) \leq n} \langle A, E_{S,T}\rangle E_{S,T}/||E_{S,T}^{\lambda}||^2$$
\end{proposition}

Before proving the main result of this section, we recall the following fact concerning integrating a 
function $f : G \rightarrow V$, where $G$ is a compact topological group and $V$ is a Hilbert space:

\begin{equation*}
	\langle \int_G f(g)d\mu(g) , w \rangle = \int_G \langle f(g) , w \rangle d\mu(g)
\end{equation*}

\begin{theorem}
\label{theorem9}
The following integration formula holds true
$$\int_{U_n}u_{i_1j_1}\ldots u_{i_dj_d}\overline u_{k_1l_1}\ldots \overline u_{k_dl_d}d\mu (U) =  
\sum_{\lambda \vdash d ; S,T \in Stab(\lambda), l(\lambda) \leq n} \dfrac{\langle e_{J,L},E_{S,T}^{\lambda}\rangle \langle E_{S,T}^{\lambda},e_{I,K} \rangle}{||E_{S,T}^{\lambda}||^{2}}.$$
\end{theorem}

\begin{proof}
We begin with noting that 
$$\int_U e_{K,I}U^{\otimes d}e_{J,L}{U}^{* \otimes d} d\mu(U) = 
\sum_M\langle \mathbb{E}(e_{J,L}), e_{I,M}\rangle e_{K,M}.$$
Indeed,
$\int_U E_{K,I}U^{\otimes d}E_{J,L}{U}^{*\otimes d} d\mu(U)
=\sum_{R,M} \langle \int_U E_{K,I}U^{\otimes d}E_{J,L}{U}^{* \otimes d} d\mu(U), E_{R,M}\rangle E_{R,M}$, 
 and one notes that 
 \begin{align*}
 \langle \int_U e_{K,I}U^{\otimes d}e_{J,L}{U}^{* \otimes d} d\mu(U), 	
 e_{R,M} \rangle = 
 \int_U \langle  e_{K,I}U^{\otimes d}e_{J,L}{U}^{* \otimes d}, e_{R,M} \rangle d\mu(U)  = \\
 \langle \int_U U^{\otimes d}e_{J,L}{U}^{* \otimes d}, e_{I,K}e_{R,M} d\mu(U) \rangle  = 
 \langle \mathbb{E}(e_{J,L}), e_{I,M} \rangle \delta_{K,R}.
 \end{align*}
  Taking the trace gives the first equality. The second comes from decomposing $\mathbb{E}(e_{J,L})$ in terms of the basis $E_{T,S}^{\lambda}$ from the previous section.
 \end{proof}
 \subsection{Examples}
Now, we investigate a few examples. 

First we start with $\int |u_{1,1}|^2d\mu(U)$.
One has $\int |u_{1,1}|^2d\mu(U) = \langle \mathbb{E}(e_{1,1}), e_{1,1}\rangle = \langle \dfrac{1}{n}$Id$, 
e_{1,1}\rangle = \dfrac{1}{n}$.  

Next, we look at $\int |u_{1,1}|^4 d\mu(U)$.
Here, we have two matrix units in $\mathbb{C}[S_2]$, namely 
$E_T = \dfrac{1}{2}(1+(1,2))$ and $E_S = \dfrac{1}{2}(1-(1,2))$. 

In $M_n \otimes M_n$, these elements 
are 
$$\dfrac{1}{2}(\Id \otimes \Id+\sum_{i,j} e_{i,j}\otimes e_{j,i}),\dfrac{1}{2}(\Id \otimes \Id - \sum_{i,j} e_{i,j}\otimes e_{j,i}).$$ 
Therefore,
\begin{eqnarray*}
\int |u_{1,1}|^4d\mu(U)&=&\langle \mathbb{E}({e_{1,1}\otimes e_{1,1}}), e_{1,1}\otimes e_{1,1}\rangle \\
&=&||E_T||^{-2} \langle e_{1,1}\otimes e_{1,1}, E_T\rangle + ||E_S||^{-2} \langle e_{1,1}\otimes e_{1,1}, E_S\rangle\\
&=&  ||E_T||^{-2} \langle e_{1,1}\otimes e_{1,1},\dfrac{1}{2}(\Id \otimes \Id+\sum_{i,j} e_{i,j}\otimes e_{j,i})\rangle + \\
&& ||E_S||^{-2} \langle e_{1,1}\otimes e_{1,1}, \dfrac{1}{2}(\Id \otimes \Id-\sum_{i,j} e_{i,j}\otimes e_{j,i})\rangle.
\end{eqnarray*}

Next we calculate the quantities $||E_T||^{2}$ and $||E_S||^{2}$:

\begin{eqnarray*}
||E_T||^2 &=& \dfrac{1}{4}\langle \Id\otimes \Id + \sum_{i,j}e_{i,j}\otimes e_{j,i}, \Id\otimes \Id + \sum_{i,j}e_{i,j}\otimes e_{j,i} \rangle \\
&=&\dfrac{1}{4}(\langle \Id\otimes \Id, \Id\otimes \Id\rangle + 2\langle \Id\otimes \Id, \sum_{i,j}e_{i,j}\otimes e_{j,i}\rangle + \\
&& \sum_{i,j,k,l}\langle e_{i,j}\otimes e_{j,i}, e_{k,l} \otimes e_{l,k}\rangle )\\
&=& \dfrac{1}{4}(n^2 + 2n + n^2) = \dfrac{n(n+1)}{2}\\
\end{eqnarray*}

\begin{eqnarray*}
||E_S||^2 &=& \dfrac{1}{4}\langle \Id\otimes \Id - \sum_{i,j}e_{i,j}\otimes e_{j,i}, \Id\otimes \Id - \sum_{i,j}e_{i,j}\otimes e_{j,i} \rangle \\
&=&\dfrac{1}{4}(\langle \Id\otimes \Id, \Id\otimes \Id\rangle - 2\langle \Id\otimes \Id, \sum_{i,j}e_{i,j}\otimes e_{j,i}\rangle + \\
&& \sum_{i,j,k,l}\langle e_{i,j}\otimes e_{j,i}, e_{k,l} \otimes e_{l,k}\rangle )\\
&=& \dfrac{1}{4}(n^2 - 2n + n^2) = \dfrac{n(n-1)}{2}\\
\end{eqnarray*}

A final calculation shows that 
\begin{eqnarray*}
\langle e_{1,1}\otimes e_{1,1},\dfrac{1}{2}(\Id \otimes \Id+\sum_{i,j} e_{i,j}\otimes e_{j,i})\rangle =\\
 \dfrac{1}{2}(\langle e_{1,1}\otimes e_{1,1}, \Id\otimes \Id \rangle + \langle e_{1,1}\otimes e_{1,1}, \sum_{i,j} e_{i,j}\otimes e_{j,i}\rangle) = 
 \dfrac{1}{2}(1+1) = 1,\\
\end{eqnarray*}
and similarly, $\langle e_{1,1}\otimes e_{1,1},\dfrac{1}{2}(\Id \otimes \Id-\sum_{i,j} e_{i,j}\otimes e_{j,i})\rangle = \dfrac{1}{2}(1-1) = 0$. 

Hence we arrive at $\int |u_{1,1}|^4d\mu(U) = \dfrac{2}{n(n+1)}$.

\subsection{Integrating over Corners}
\label{subsec:corners}

Here, we consider the following problem: suppose than in the integrals from the previous section we 
assume that all indices $I,J,K,L \in \{1,\dots, k\}^d$ for $k \leq n$. 
We can prove the following:

\begin{theorem}
\label{theorem10}
Suppose that $I,J,K,L \in \{1,\dots,k\}^d$. 
Then we have the following:
$$\int_{U_n}u_{i_1j_1}\ldots u_{i_dj_d}\overline u_{k_1l_1}\ldots \overline u_{k_dl_d}d\mu (U) =  
\sum_{\lambda \vdash d ; S,T \in Stab(\lambda), l(\lambda) \leq \min{\{k,d\}}} \dfrac{\langle e_{J,L},E_{S,T}^{\lambda}\rangle \langle E_{S,T}^{\lambda},e_{I,K} \rangle}{||E_{S,T}^{\lambda}||^{2}}.$$ 
In other words, we need only sum over those $\lambda$ with $l(\lambda) \leq k$.

\end{theorem}

Recall that $p_n^d: \mathbb{C}[S_d] \rightarrow \End((\mathbb{C}^n)^{\otimes d})$ is injective on 
$\mathbb{C}_n[S_d] := \bigoplus_{l(\lambda) \leq n}\End(V_{\lambda})$, hence is injective when 
restricted to $\mathbb{C}_k[S_d]$ where $1 \leq k \leq n$. The integration question considered in this section 
can be reformulated in the following way: there is a natural embedding of $M_k^{\otimes d}$ in 
$M_n^{\otimes d}$ for $1 \leq k \leq n$, where we view a $k \times k$ matrix as the upper left corner of 
an $n \times n$ matrix, and extend this inclusion to tensors. 

\begin{proof}[Proof of Theorem \ref{theorem10}]
We need to show that the image of $M_k^{\otimes d}$ under conditional expectation 
$\mathbb{E} : M_n^{\otimes d} \rightarrow \mathbb{C}_n[S_d]$ is orthogonal to the sum of all 
$\End(V_\lambda)$ where $l(\lambda) > k$. In particular, if $A \in M_k^{\otimes d}$ and 
$B \in \End(V_{\lambda})$ for $l(\lambda) > k$ we have that $\langle A, B \rangle = 0$. 

First note that the inclusion $i: (\mathbb{C}^k)^{\otimes d} \rightarrow  (\mathbb{C}^n)^{\otimes d}$ 
commutes with the action of $S_d$, that is $p_n^d(\sigma)i(x) = i(p_k^d(\sigma)x)$ for 
$x \in (\mathbb{C}^k)^{\otimes d}$. From this it follows that $(\mathbb{C}^k)^{\otimes d}$ lies entirely 
in the irreducible component of $(\mathbb{C}^n)^{\otimes d}$ consisting of the $V_{\lambda}$ with 
$l(\lambda) \leq k$. In particular, if $x \in (\mathbb{C}^k)^{\otimes d}$ and $B \in \End(V_{\lambda}) \subset M_n^{\otimes d}$ 
with $l(\lambda) > k$ we have $Bx = 0$. But the inner product $\langle A, B \rangle$ can be written as 
$\sum_I \langle Ae_I, Be_I \rangle$. For $e_I \in (\mathbb{C}^k)^{\otimes d}$ we have $Be_I = 0$. 
Further, for any $e_I$ not in $(\mathbb{C}^k)^{\otimes d}$ we have $Ae_I = 0$, since $A$ lies in the span of 
those $e_{J,K} \in M_k^{\otimes d}$. Hence every term in this sum is $0$, and the inner product is $0$ as required. This completes the proof.
\end{proof}

The proof of theorem \ref{theorem10}, 
together with the fact that 
$\langle \mathbb{E}(e_{J,L}), e_{I,K} \rangle = \langle e_{J,L}, \mathbb{E}(e_{I,K}) \rangle$ 
shows that in the integral $\langle \mathbb{E}(E_{J,L}), E_{I,K} \rangle$ we need only sum over 
diagrams of length $k$ or less, where $k$ is the minimum of max$\{j_1, \dots, j_d, l_1, \dots, l_d \}$ 
and \\ max$\{i_1, \dots, i_d, k_1, \dots, k_d \}$. In this way we have simplified formulas for integrating 
over not just upper left hand squares in $U_n$,  but for integrating over upper left hand rectangles 
in $U_n$, in terms of the lengths of Young diagrams.

Note that in the orthogonal case, a formula was obtained in \cite{bcs2}. While there is no overlap since the integration group is different,
it is worth observing that this approach could be conducted to obtain
a conceptual proof for the paper \cite{bcs2}.

Applying Theorem \ref{theorem10}
to the case of one row,  we obtain the following formula:

\begin{theorem}
\label{theorem11}
Let $J = \{j_1, \dots, j_d \}$ and $L = \{l_1, \dots, l_d \}$ be arbitrary indices. Then the integral $\int_U u_{1, j_1} \dots u_{1, j_d} \overline{u_{1, l_1}} \dots \overline{u_{1, l_d}} d\mu(U)$ is equal to $\dfrac{k}{d!|E_T|^2}$ where $T$ is the unique standard filling of $\lambda = (d)$ and $k$ is the number of permutations in $S_d$ mapping $J$ to $L$.
\end{theorem}

Note that when $d = 2$ and $J = L = (1,1)$ we recover the formula from the calculation of $\int_U |u_{1,1}|^4d\mu(U)$ on the previous page. 

Let us mention a related formula that was known for $\int_U u_{1, j_1} \dots u_{1, j_d} \overline{u_{1, l_1}} \dots \overline{u_{1, l_d}} d\mu(U)$
(cf \cite{mat}, Proposition 2.4 for a direct proof with Weingarten calculus). 
Our proof in this paper is new and has the potential for generalization to more than one row.
Let us outline below yet an other proof, of probabilistic nature:
 
First, observe that with probability one, the random vector  $(u_{1,1},\ldots ,u_{1,n})$ has the same distribution
as $(X_1/(\sum_{i=1}^d |X_i|^2)^{1/2}, \ldots, X_d/(\sum_{i=1}^d |X_i|^2)^{1/2})$, where $(X_1,\ldots ,X_d)$ are i.i.d. standard complex valued 
gaussians distributions. Further, $\sum_{i=1}^n |X_i|^2$ is independent from 
$(X_1/\sum_{i=1}^d |X_i|^2, \ldots, X_d/\sum_{i=1}^d |X_i|^2)$. Putting this together, 
we obtain the following formula (cf \cite{mat}, Proposition 2.4):

$$\int_U u_{1, j_1} \dots u_{1, j_d} \overline{u_{1, l_1}} \dots \overline{u_{1, l_d}} d\mu(U)=\frac{\prod_i r_i!}{n(n+1)\ldots (n+d-1)},$$
where $r_1,r_2\ldots $ are the number of elements in the blocks induced by the partition $i\to j_i$ assuming that these numbers
are the same up to permutation, that is if $i\to j_i$ is replaced by $i\to l_i$.  To make this more precise, we say that two $d-$indices 
$K$ and $J$ containing values in $\{1,\dots,n\}$ are \emph{of the same type} if each index set contains the same number of occurances of each of 
$1, 2, \dots, n$. For example, the indices $(1,1,2,2,5,5,5)$ and $(5,1,2,1,2,5,5)$ are of the same type, while $(1,1,2,2,5,5,5)$ 
and $(1,2,2,2,5,5,5)$ are not. Equivalently, $I$ and $J$ are of the same type iff one can be obtained by another by a permutation 
$\sigma \in S_d$. This gives an equivalence relation of the set $[n]^{[d]}$ of functions 
$I: \{1, \dots, d \} \rightarrow \{1, \dots, n \}$. The proposition above says that if $J$ and $L$ lie in different equivalence classes, 
the integral $\int_U u_{1, j_1} \dots u_{1, j_d} \overline{u_{1, l_1}} \dots \overline{u_{1, l_d}} d\mu(U)$ is zero, otherwise we get 
the formula above. Reformulating this using  theorem \ref{theorem11} we get the following:

\begin{corollary}
\label{corollary1}
If two indices $J$ and $L$ are not of the same type, then the integral 
$\int_U u_{1, j_1} \dots u_{1, j_d} \overline{u_{1, l_1}} \dots \overline{u_{1, l_d}} d\mu(U)$ is zero.
\end{corollary}

Another formulation of the integral in Theorem \ref{theorem11} is that 

$$ \int_U u_{1, j_1} \dots u_{1, j_d} \overline{u_{1, l_1}} \dots \overline{u_{1, l_d}} d\mu(U) =  \dfrac{d!\sum_{\sigma}\delta_{\sigma J, L}}{\sum_{\sigma, \beta} \sum_I \delta_{\sigma I, \beta I}}$$

This follows by a simple calculation of $|E_T|^2 = \langle E_T, E_T \rangle$ for $E_T = \dfrac{1}{d!} \sum_{\sigma} \sum_I E_{\sigma I , I} \in M_n^{\otimes d}$.

Finally, note that a similar analysis could be performed for two rows or more. The notation to obtain a closed formula is already
quite cumbersome at that level and will be studied elsewhere.
Note that similar results were obtained in the orthogonal case by \cite{bcs2}. 

\section*{ Acknowledgements}

T.C.'s research was supported by NSERC and OGS fellowships.
T.C. and B.C. were supported by NSERC discovery grant and an ERA grant.
B.C. was partly supported by AIMR funding.

We would like to thank Sho Matsumoto, Jean-Marc Schlenker, and Teo Banica for their suggestions of improvements on a preliminary
version of this paper.  

\bibliographystyle{alpha}
\bibliography{biblio}

\def\cprime{$'$}
\begin{thebibliography}{CST10}

\bibitem[BC07a]{bc1}
Teodor Banica and Beno{\^{\i}}t Collins.
\newblock Integration over compact quantum groups.
\newblock {\em Publ. Res. Inst. Math. Sci.}, 43(2):277--302, 2007.

\bibitem[BC07b]{bc2}
Teodor Banica and Beno{\^{\i}}t Collins.
\newblock Integration over quantum permutation groups.
\newblock {\em J. Funct. Anal.}, 242(2):641--657, 2007.

\bibitem[BC08]{bc3}
T.~{Banica} and B.~{Collins}.
\newblock {Integration over the Pauli quantum group}.
\newblock {\em Journal of Geometry and Physics}, 58:942--961, August 2008.

\bibitem[BCS11]{bcs2}
Teodor Banica, Benoit Collins, and Jean-Marc Schlenker.
\newblock On polynomial integrals over the orthogonal group.
\newblock {\em J. Combin. Theory Ser. A}, 118(3):778--795, 2011.

\bibitem[CM09]{cm}
Beno{\^{\i}}t Collins and Sho Matsumoto.
\newblock On some properties of orthogonal {W}eingarten functions.
\newblock {\em J. Math. Phys.}, 50(11):113516, 14, 2009.

\bibitem[Col03]{co}
Beno{\^{\i}}t Collins.
\newblock Moments and cumulants of polynomial random variables on unitary
  groups, the {I}tzykson-{Z}uber integral, and free probability.
\newblock {\em Int. Math. Res. Not.}, (17):953--982, 2003.

\bibitem[C{\'S}06]{cs1}
B.~{Collins} and P.~{{\'S}niady}.
\newblock {Integration with Respect to the Haar Measure on Unitary, Orthogonal
  and Symplectic Group}.
\newblock {\em Communications in Mathematical Physics}, 264:773--795, June
  2006.

\bibitem[CST10]{Tolli}
Scarabotti~F. Ceccherini-Silberstein, T. and F.~Tolli.
\newblock {\em Representation Theory of the Symmetric Groups: The
  Okounkov-Vershik Approach, Character Formulas, and Partition Algebras}.
\newblock Cambridge University Press, New York, New York, 2010.

\bibitem[Mat13]{mat}
Sho Matsumoto.
\newblock Moments of a single entry of circular orthogonal ensembles and
  {W}eingarten calculus.
\newblock {\em Lett. Math. Phys.}, 103(2):113--130, 2013.

\bibitem[RW92]{Ram}
A.~Ram and H.~Wenzl.
\newblock {\em Matrix Units for Centralizer Algebras}.
\newblock MR1144939 (93g:16024), 1992.

\bibitem[VO05]{Vershik}
A.~M. Vershik and A.~Yu. Okounkov.
\newblock {A New Approach to the Representation Theory of the Symmetric Groups
  II. (Russian)}.
\newblock {\em Zap. Nauchn. Sem. S.-Peterburg. Otdel. Mat. Inst. Steklov.
  (POMI)}, 307 (2004), (Teor. Predst. Din. Sist. Komb. i Algoritm. Metody.
  10):57Ð98, 281; translation in J. Math Sci. (N.Y.) 131, no. 2, 5471Ð5494.,
  2005.

\bibitem[Wei78]{We}
Don Weingarten.
\newblock Asymptotic behavior of group integrals in the limit of infinite rank.
\newblock {\em J. Mathematical Phys.}, 19(5):999--1001, 1978.

\end{thebibliography}

\end{document}